\documentclass[12pt,a4paper]{amsart}
\usepackage{mathrsfs}
\usepackage{amssymb,amsmath,amsthm,color}
\usepackage{graphicx, cite, txfonts}
\usepackage{hyperref}
\usepackage{url}
\usepackage{setspace}
\usepackage{enumerate}

\newtheorem{theorem}{Theorem}[section]
\newtheorem{lemma}[theorem]{Lemma}
\newtheorem{proof of lemma}[theorem]{Proof of Lemma}
\newtheorem{proposition}[theorem]{Proposition}

\newtheorem{corollary}[theorem]{Corollary}

\theoremstyle{definition}
\newtheorem{definition}[theorem]{Definition}
\newtheorem{question}[theorem]{Question}
\newtheorem{res}[theorem]{Result}

\usepackage[headheight=110pt,top=1.4in, bottom=1.4in, left=1.0in, right=1.0in]{geometry}
\hypersetup{colorlinks,linkcolor={red},citecolor={red},urlcolor={blue}}

\newtheorem{remark}[theorem]{Remark}

\numberwithin{equation}{section}
\makeatletter
\@namedef{subjclassname@2020}{\textup{2020} Mathematics Subject Classification}
\makeatother
\newcommand{\bigslant}[2]{{\raisebox{.2em}{$#1$}\left/\raisebox{-.2em}{$#2$}\right.}}

\begin{document}

%%%%% To ease editing, for IMPAN journals, add:

\baselineskip=17pt

%%%%%%%%%%%%%%%%

\title[] { Projection from space of \emph{two-Lipschitz} operators onto the space of Bilinear maps }
\author[Arindam Mandal]{Arindam Mandal$^\dagger$}
\address{Arindam Mandal, School of Mathematical Sciences, National Institute of Science Education and Research
Bhubaneswar, An OCC of Homi Bhabha National Institute, P.O. Jatni, Khurda, Odisha 752050,
India.}
\email{arindam.mandal@niser.ac.in}
\thanks{The author is supported by a research fellowship from the Department of Atomic Energy (DAE), Government of India.}
\thanks{$\dagger$ \tt{Corresponding author}}
\date{}

\begin{abstract} 
In this article, we establish the existence of a norm-one projection from the space of all \emph{two-Lipschitz} operators onto the space of all bounded bilinear operators under certain conditions on the corresponding codomain spaces, using the method of invariant means. We also show that, when the codomain is an injective Banach space, the quotient of the \emph{two-Lipschitz} operator space by the bounded bilinear space is isometrically isomorphic to a specific operator space, via vector-valued duality. We conclude by proving a necessary and sufficient condition for a \emph{two-Lipschitz} operator to be a bilinear map. As an application of the theory developed here, we present an alternative proof that $\bigslant{L^{\infty}(\mathbb{R}\times \mathbb{R})}{span\{\textbf{1}\}}$ is a dual space.
\end{abstract}


\maketitle
\pagestyle{myheadings}
\markboth{}{\emph{two-Lipschitz} Operators: Complemented Subspaces and Quotient Representations}
\section{\bf Introduction and Preliminaries}
Over the years, Lipschitz spaces have attracted significant research interest in the field of nonlinear functional analysis. Let $X$ be a real normed linear space, and let $Lip_0(X, \mathbb{R})$ be the space of all real-valued Lipschitz maps that vanish at $0$. A notable development  occurred in $1964$, when Joram Lindenstrauss in \cite[Theorem 2]{ONPIBS}  studied the existence of a projection map from $Lip_0(X, \mathbb{R})$ onto the space of all real-valued bounded linear functionals, denoted by 
$X^{\ast}$. Now the linear dual $X^{\ast}$ is a closed subspace of the Lipschitz dual $Lip_0(X, \mathbb{R})$. Then
 as a consequence of Lindenstrauss's result (also see \cite[Proposition 7.5, p 173]{GNFA}), it is immediate that the $X^{\ast}$ is complemented in $Lip_0(X, \mathbb{R})$. This result highlights a connection between linear and Lipschitz functionals, bridging linear theory and Lipschitz geometry. Subsequent studies revealed that $L(X,Y)$, the space of bounded linear maps between Banach spaces $X$ and $Y$, forms a closed subspace of $Lip_0(X,Y)$ under the natural Lipschitz norm:
$$\|T\|=Lip(T):= \sup\limits_{x,z \in X;\ x \neq z} \frac{\|Tx-Tz\|}{\|x-z\|}.$$
However, the question of whether $L(X,Y)$ is complemented in $Lip_0(X,Y)$ remained unaddressed untill recently, when Karn and Mandal have studied this in \cite{karn2025positionlxylip0x}. More recently, several authors have expanded the notion of Lipschitz spaces to a two-variable framework, initially in $2009$, Dubei et al. introduced in \cite{FBSAEOLM} the definition of \emph{two-Lipschitz} maps which is defined on the
cartesian product of two pointed metric spaces (a metric space with a distinguished point) with values in a Banach space, that is Lipschitz separately
in each variable. Under some diﬃcult adequate requirements, it is shown that every \emph{two-Lipschitz} mapping
is associated with a continuous bilinear mapping from the product of two suitable free Banach spaces to another
Banach space (see \cite{FBSAEOLM}). This idea worked successfully without any conditions in \cite{PSGBBMAD}. There, Sánchez-Pérez
present a suitable definition of real-valued \emph{two-Lipschitz} mappings (under the name of Lipschitz bi-forms)
that admits a good continuous bilinearization between Banach spaces. Further, Hamidi et al. \cite{TLOI} introduced a new concept of \emph{two-Lipschitz} operator ideals between pointed metric spaces
and Banach spaces. They extended to the two-Lipschitz mappings setting a linear procedure for creating ideals of
\emph{two-Lipschitz} operators from a given linear operator ideal in \cite{TLOI}. Several recent works have been carried out in this area; see, for example, \cite{BIOTLOBMABS}, \cite{FOTLIO}, \cite{LCIHIALII}, \cite{LLOATCTLTP}. Throughout this article, we adopt the definition of \emph{two-Lipschitz} operators given by Hamidi et al. in \cite{TLOI}.

One of our main purposes is to study the complementation of the space of bilinear maps between Banach spaces within the space of \emph{two-Lipschitz} maps between Banach spaces. We also establish a non-degenerate vector-valued dual pairing between the space of \emph{two-Lipschitz} operators and the projective tensor product of two suitable Lipschitz free spaces. To begin with, the author's concern regarding one of the questions addressed in this article, we first consider the following discussion.
\subsection{\emph{Two-Lipschitz} operators in contrast with bounded bilinear maps}
Let $X, Y$, and $E$ be real Banach spaces. A map $T$ from $X \times Y$ into $E$ is said to be \emph{two-Lipschitz} if there exists an absolute constant $k > 0$ such that $$\|T(x,y)- T(x',y)-T(x,y')+T(x',y')\| \le k \|x - x'\| \cdot \|y-y'\|$$ for all $x, x' \in X$ and $y,y' \in Y$. 

 For the collection of all such mappings, we write $BLip(X, Y; E)$. Then it is immediate that $BLip(X, Y; E)$ forms a vector space with respect to pointwise addition and scalar multiplication. Our main focus is to study those \emph{two-Lipschitz} mappings $T$ belongs to $BLip(X, Y; E)$, satisfying $T(x,0)=0$ for all $x \in X$ and $T(0, y)=0$ for all $y\in Y$, that is. $BLip_0(X, Y; E) = \left\lbrace T\in BLip(X, Y): T(x,0)=0=T(0,y)\ \forall x \in X, \forall \ y\in Y\right\rbrace.$
 It is well known that
$$BLip(T)=\sup \left\lbrace \frac{\Vert T(x,y)- T(x',y)-T(x,y')+T(x',y') \Vert}{\|x-x'\|\cdot\|y-y'\|}: x,x' \in X, y,y' \in Y; x \ne x', y\ne y' \right\rbrace,$$ is a norm on $BLip_0(X, Y; E)$, and with respect to $BLip(\cdot)$, $BLip_{0}(X,Y; E)$ forms a Banach space. 

The space 
$BLip_0(X, Y; E)$ is often viewed as a natural nonlinear extension of
$Blin(X,Y;E)$, the space of bounded bilinear operators from 
$X\times Y$ into 
$E$. It can be easily seen that 
$Blin(X,Y;E)$ embeds as a subspace of $BLip_0(X,Y ; E)$. The bilinear space 
$Blin(X,Y;E)$ is equipped with the standard operator norm:
$$\| S \| = \sup \left\{ \frac{\| S(x, y) \|}{\|x\|\cdot\|y\|}: x \in X\setminus\{0\} ~ \mbox{and} ~ y \in Y\setminus\{0\} \right\}$$ 
for any $S \in Blin(X, Y; E)$. Moreover, in view of the definition of $BLip(\cdot)$ one can easily see for $S \in Blin(X, Y; E)$, 
$\|S\|=BLip(S)$,  so that $Blin(X, Y; E)$ is a closed subspace of $BLip_0(X, Y; E)$ (see \cite[Remark 2.4]{TLOI}). Based on the discussion above  and motivated by Lindenstrauss's result, the authors are led to consider the following question:
\begin{question} \label{qns1}
	Is $Blin(X, Y; E)$ complemented in $BLip_0(X, Y; E)$?
\end{question}
 We show that there is a surjective linear projection 
$P : BLip_0(X, Y; E) \xrightarrow{} Blin(X, Y; E)$ with norm one, whenever $E$ is a dual space. The proof technique relies on the existence of a vector-valued invariant mean (a vector-valued bounded linear map from the space of bounded functions on a normed space; details have been discussed later), which is guaranteed whenever the codomain space $E$ is a dual space. Given a \emph{two-Lipschitz} map
$T \in BLip_0(X, Y; E)$, we construct a suitable function in
$\ell^{\infty}( X\times Y, E)$ and then apply the vector-valued invariant mean to the bounded function derived from
$T$. The existence of such a mean enables us to affirmatively answer Question \ref{qns1} in the case where
$E$ is a dual space. However, the general case remains open. 
In summary, when 
$E$ is a dual space, 
$Blin(X, Y; E)$ is a complemented subspace of 
$BLip_{0}(X, Y; E)$, since 
$BLip_{0}(X, Y; E)=Blin(X, Y; E) \oplus_{1} \ker(P)$ with equivalent norms, implying that 
$\bigslant{BLip_{0}(X, Y; E)}{Blin(X, Y; E) }$ is topologically isomorphic to 
$\ker(P)$.

\subsection{General quotient structure of \emph{two-Lipschitz} space}
However, the authors are further interested in addressing the following question:
\begin{question} \label{qns2}
Is this quotient space not only topologically but also isometrically isomorphic to a concrete space of operators?
\end{question}
In response, we show that when 
$E$ is an injective Banach space; the quotient is indeed isometrically isomorphic to a well-defined operator space. Moreover, we identify additional quotient spaces of 
$BLip_0(X, Y; E)$ that admit isometric isomorphisms with certain space of operators. 
To proceed further, let us recall the construction of the Lipschitz free space $F(X)$ (introduced by Godefroy and Kalton in \cite{LFBS}).  
For any $x \in X$, we denote by $\delta_x$ the evaluation functional, i.e. $\delta_x(f)=f(x)$ for $f \in Lip_0(X, \mathbb{R})$. It
is easy to see that $\delta_X : X \to span\{\delta_x : x \in X\}; x \mapsto \delta_x$ is an isometric (nonlinear) embedding of $X$ into $Lip_0(X, \mathbb{R})^*$. The space
$F(X)$ is deﬁned to be the closed linear span of $\{\delta_x: x \in X\}$ with the dual
space norm denoted simply by $\|\cdot\|$. It is known that $F(\mathbb{R})$ is isometrically isomorphic to $L^1(\mathbb{R})$, we denote it as $F(\mathbb{R}) \cong L^1(\mathbb{R})$ (see \cite[p 128]{LFBS}). Then we consider a linear contraction $\beta_X: F(X) \to X$, a left inverse of the Lipschitz map $\delta_X$. For details and additional properties, see \cite{OTSOLFS}, \cite{LA}.
Furthermore, we establish the following more general result, which provides a partial answer to the above-posed Question \ref{qns2}:
\begin{res} \label{so quotient of Blip}
    For any subspace $\mathcal{D}$ of $F(X)\hat{\otimes}_{\pi}F(Y)$ and $E$ is an injective Banach space, $$\bigslant{BLip_{0}(X, Y;E)}{^{\diamondsuit}\mathcal{D} } \cong  L(\mathcal{D}, E).$$
\end{res}
 Where $F(X)\hat{\otimes}_{\pi}F(Y)$ be the complete projective tensor product of $F(X)$ and $F(Y)$ (with the norm denoted as $\|\cdot\|_{\pi}$) and $ ^{\diamondsuit}\mathcal{D}:= \{ T \in BLip_{0}(X, Y; E) : T_L(\xi)=0 ~ \forall \xi \in \mathcal{D}\}$ ($T_L$ be unique linearization map corresponding to $T \in BLip_0(X,Y; E)$). In the special case for $X, Y$ and $E$ to be $\mathbb{R}$, using this result, we obtain an alternative proof of the following  consequence:
$$\bigslant{L^{\infty}(\mathbb{R} \times \mathbb{R})}{\{ \mbox{set of all constant maps on}\ \mathbb{R} \times \mathbb{R} \}}$$ is isometrically isomorphic to a dual space. Moreover, we explicitly identify its predual towards the end of this article.
  For details on bilinear maps and tensor products, we refer to \cite{BMATPIOT}. Through the study of particular aspects of $BLip_0(X, Y; E)$ and $F(X)\hat{\otimes}_{\pi}F(Y)$, we also provide a necessary and sufficient condition under which a \emph{two-Lipschitz} operator is bilinear.

The rest of the paper is organized as follows. In Section \ref{sec2}, we provide the proof of one of the main results, Theorem \ref{proj}, which offers a partial answer to Question \ref{qns1}. In Section $\ref{sec3}$ we study an $E$-valued dual action between $BLip_0(X, Y; E)$ and $F(X)\hat{\otimes}_{\pi}F(Y)$. Section \ref{sec4} presents several general results concerning the quotients of \emph{two-Lipschitz} operator spaces. Also, Remark \ref{four rem }$(2)$ answers the Question \ref{qns2} whenever $E$ is an injective Banach space. Section \ref{sec5} provides some nontrivial examples of the quotient space. 
\section{Embedding of $Blin(X,Y; E)$ into $BLip_0(X, Y; E)$} \label{sec2}
Let $X, Y$, and $E$ be real Banach spaces. We again recall that for $T \in Blin(X, Y; E)$, we have $\Vert T \Vert = BLip(T)$ so that $Blin(X, Y)$ is a closed subspace of $BLip_0(X, Y; E)$. This section proves that $Blin(X, Y; E)$ is a complemented subspace of $BLip_0(X, Y; E)$ whenever $E$ is a dual Banach space. In this case, we show that a projection exists from $BLip_0(X, Y; E)$ onto $Blin(X, Y; E)$. 

Given a map $f : X \times Y \xrightarrow{} E$, 
we define its translation by a point 
$(p,q) \in X \times Y$ as the map 
$f_{p,q} : X \times Y \xrightarrow{} E$, where
$$f_{p,q}(x,y):= f(p+x, q+y)$$ 
for all $(x,y) \in X \times Y$. We begin with the following result. 
\begin{lemma} \label{lem1}
    Let $T \in BLip_0(X,Y; E)$, $w, w' \in X$ and $z,z' \in Y$. Consider $T_{p,q} \ $ ; translation of $T$ by any $(p,q) \in X \times Y$; that is $T_{p,q}(x,y)= T(p+x, q+y)$ for all $(x,y) \in X \times Y$. Then the following hold:
    \begin{enumerate}
        \item  $\phi_T^{w,z}$ defined by $\phi_T^{w,z}= T_{w,z} - T_{w,0}-T_{0,z}+ T$ belongs to $\ell^{\infty}(X \times Y; E)$.
        \item $\phi_T^{w+w',z}= \left(\phi_T^{w,z}\right)_{w',0} + \phi_T^{w',z}$ and $\phi_T^{w,z+ z'} = \left(\phi_T^{w,z}\right)_{0,z'} + \phi_T^{w,z'}$.
        \item $\left(\phi_T^{-w,z}\right)_{w,0}= -\phi_T^{w,z}$.
        \item $ \phi_T^{2w,z}= \left(\phi_T^{w,z}\right)_{w,0} + \phi_T^{w,z}$ and $\phi_T^{w,2z}= \left(\phi_T^{w,z}\right)_{0,z} + \phi_T^{w,z}$.
    \end{enumerate}
    \end{lemma}
\begin{proof}
Let $T \in BLip_0(X,Y; E)$ and fix $(w,z) \in X \times Y$. Then for any $(x,y) \in X\times Y$
\begin{eqnarray*}
     \|\phi_T^{w,z}(x,y)\| &=& \|T_{w,z}(x,y) - T_{w,0}(x,y)-T_{0,z}(x,y)+ T(x,y)\|\\
     &=& \|T(w+x, z+y)-T(w+x, y)-T(x, z+y)+ T(x, y)\|\\
     &\leq& BLip(T)\|w\|\|z\|.
\end{eqnarray*}
   Therefore, $\phi_T^{w,z} \in \ell^{\infty}(X \times Y; E)$ with $\|\phi_T^{w,z} \|_{\infty} \leq BLip(T)\|w\|\|z\|$.

   Let $w, w' \in X$ and $z,z' \in Y$. Then for any $(x,y) \in X \times Y$
   \begin{eqnarray*}
       \left(\phi_T^{w,z}\right)_{w',0} (x,y) &=& \left(T_{w,z} - T_{w,0}-T_{0,z}+ T\right)_{w',0}(x,y)\\
       &=& T(w+w'+x, z+y) - T(w+w'+x, y) - T(w'+x, z+y) + T(w'+x, y)\\
       &=& \left(T_{w+w',z}- T_{w+w',0}-T_{w',z}+ T_{w',0}\right)(x,y).
   \end{eqnarray*}
   This implies that $\left(\phi_T^{w,z}\right)_{w',0}= T_{w+w',z}- T_{w+w',0}-T_{w',z}+ T_{w',0}$. Further
   \begin{eqnarray*}
       \phi_T^{w+w',z} &=& T_{w+w',z} - T_{w+w',0}-T_{0,z}+ T\\
       &=& \underbrace{T_{w+w',z} - T_{w+w',0} - T_{w',z} +T_{w',0}} + \underbrace{T_{w',z} -T_{w',0} -T_{0,z}+ T}\\
       &=& \left(\phi_T^{w,z}\right)_{w',0} + \phi_T^{w',z}.
   \end{eqnarray*}
   In a similar manner we have $\phi_T^{w,z+ z'} = \left(\phi_T^{w,z}\right)_{0,z'} + \phi_T^{w,z'}$. 
   
   Again for any $(x,y) \in X \times Y$ 
   \begin{eqnarray*}
       \left(\phi_T^{-w,z}\right)_{w,0}(x,y)&=& \left(T_{-w,z} - T_{-w,0}-T_{0,z}+ T\right)_{w,0}(x,y)\\
       &=& T(x,y+z) -T(x,y)-T(w+x, y+z) +T(w+x,y)\\
       &=& -\left( \left(T_{w,z} - T_{w,0} -T_{0,z} +T\right)\right)(x,y)\\
       &=& -\phi_T^{w,z} (x,y).
   \end{eqnarray*}
   Hence $\left(\phi_T^{-w,z}\right)_{w,0}= -\phi_T^{w,z}$. Also,
   \begin{eqnarray*}
       \left(\phi_T^{w,z}\right)_{w,0}(x,y)&=& \left(T_{w,z} - T_{w,0}-T_{0,z}+ T\right)_{w,0}(x,y)\\
       &=& T(2w+x,y+z) -T(2w+x,y)-T(w+x, y+z) +T(w+x,y)\\
       &=&  \left(T_{2w,z} - T_{2w,0} -T_{w,z} +T_{w,0}\right)(x,y)
   \end{eqnarray*}
   Further 
   \begin{eqnarray*}
       \left(\phi_T^{2w,z}\right)(x,y)&=& \left(T_{2w,z} - T_{2w,0}-T_{0,z}+ T\right)(x,y)\\
       &=& \left(\underbrace{T_{2w,z} - T_{2w,0} -T_{w,z} +T_{w,0}} + \underbrace{T_{w,z} -T_{w,0}- T_{0,z} +T}\right)(x,y)\\
       &=& \left(\phi_T^{w,z}\right)_{w,0}(x,y) + \left(\phi_T^{w,z}\right)(x,y)
   \end{eqnarray*}
   The other part can be proved similarly. This completes the proof.
\end{proof}
Next, we recall the notion of an invariant mean. 
\begin{definition}\label{inv mean} \cite{GNFA}
	A left \emph{invariant mean} in $X$ is a linear functional $M$ on $\ell^{\infty}(X)$ (the space of all bounded maps from $X$ to $\mathbb{R}$) such that:
	\begin{enumerate}
		\item $M(\textbf{1})=1$, where $\textbf{1}$ be the constant function on $X$ that takes value $1$.
		\item $M(f)\geq 0$ for all $f \geq 0$.
		\item $M(f_x)=M(f)$ for all $x \in X$, where $f_x(z)= f(x+z)$ for all $z \in X$.
	\end{enumerate}
\end{definition}
%\subsection{Invariant mean} 
Let $Y$ be a dual Banach space. It follows from \cite[p. 417]{GNFA} that an invariant mean $M$ in $X$ induces a norm-one linear operator $\mathcal{M} :\ell^{\infty}(X, Y) \to Y$ satisfying:
\begin{enumerate} 
    \item $\mathcal{M}(f_x)=\mathcal{M}(f)$ for all $f \in \ell^{\infty}(X, Y)$ and $x \in X$. 
    \item $\mathcal{M}(\hat{y})=y$ for all $y \in Y$. Here $\hat{y} \in \ell^{\infty}(X, Y)$ is given by $\hat{y}(x) = y$ for all $x \in X$. 
\end{enumerate} 
%whenever $Y$ is a dual space. Here $\hat{y}: X \to Y$ is given by $\hat{y}(x) = y$ for all $x \in X$ 
For the proof of the existence of such an invariant mean, we refer \cite[p. 417]{GNFA}. We call this $\mathcal{M}$ as a generalized invariant mean. Further, it is easy to verify that, in particular for $Y= \mathbb{R}$, $\mathcal{M}=M$, the invariant mean.

For any two normed linear spaces $X$ and $Y$, the product space $X \times Y$ is equipped with the norm $$\|(x,y)\|_{X \times Y}= \|x\|_X + \|y\|_Y.$$
In fact $\left(X \times Y, \|.\|_{X \times Y}\right)$ becomes a Banach space whenever $X$ and $Y$ are so. We will omit the subscript in the norm notation from this point onward. The norm should be understood in the context of the respective spaces. Now, we are in a position to state our
main result of this section.
\begin{theorem} \label{proj}
    Let $X, Y$  be normed linear spaces and $E$ be a dual space. Then $ Blin(X, Y; E)$ is complemented in $BLip_0(X, Y; E)$.
\end{theorem}
\begin{proof}
    Let $X, Y$  be normed linear spaces and $E$ be a dual space. Then $X \times Y$ is a normed linear space equipped with the norm described above. Consequently, by \ref{inv mean}, there exists a generalized invariant mean $\mathcal{M}$ on $\ell^{\infty}(X \times Y, E)$.    
    
    We now define $P (T)(w,z) := \mathcal{M}\left(\phi_T^{w,z}\right)$ for each $T \in BLip_0(X,Y; E)$ and  $(w,z) \in X \times Y$, where $\phi_T^{w,z}$ is defined in the preceding lemma. 
    
    Suppose $w, w' \in X$ and $z,z' \in Y$. Then by Lemma \ref{lem1}, we obtain 
    \begin{eqnarray*}
        P (T)(w+w',z) &=& \mathcal{M}\left(\phi_T^{w+w',z}\right)\\
        &=& \mathcal{M}\left(\left(\phi_T^{w,z}\right)_{w',0} + \phi_T^{w',z}\right) \\
        &=& \mathcal{M}\left(\left(\phi_T^{w,z}\right)_{w',0}\right) +\mathcal{M}\left( \phi_T^{w',z}\right) \ (\mbox{ as} \ \mathcal{M} \ \mbox{is linear})\\
        &=& \mathcal{M}\left(\phi_T^{w,z}\right) +\mathcal{M}\left( \phi_T^{w',z}\right) \ (\mbox{ as} \ \mathcal{M} \ \mbox{is translation invariant})\\
        &=& P (T)(w,z) + P (T)(w',z).
    \end{eqnarray*}
    Using similar argument we further conclude that $$P(T)(w,z+z')= P(T)(w,z)+ P(T)(w,z').$$
    Thus, $P(T)$ is additive in each variable. Now, using a standard technique of real analysis, we show that $P(T)$ is linear in each argument. 

For $n \in \mathbb{N}$, applying Lemma~\ref{lem1}$(4)$ and the linearity of 
$\mathcal{M}$, we obtain
$$\mathcal{M}\left(\phi_T^{nw,z}\right)= \mathcal{M}\left(\left(\phi_T^{(n-1)w,z}\right)_{w,0}\right) + \mathcal{M}(\phi_T^{(n-1)w,z}).$$
    At this stage, utilizing the translation invariance of 
$\mathcal{M}$ repeatedly, we derive $$\mathcal{M}\left(\phi_T^{nw,z}\right) = n \mathcal{M}\left(\phi_T^{w,z}\right).$$
Applying Lemma \ref{lem1}$(3)$ we deduce $\mathcal{M}\left(\phi_T^{nw,z}\right) = n \mathcal{M}\left(\phi_T^{w,z}\right)$ for any $n \in \mathbb{Z}$, that is. $P (T)(nw,z)=nP (T)(w,z)$. Let $\alpha=\frac{m}{n} \in \mathbb{Q}$. Then 
$$ mP (T)(w,z)=P (T)(m w,z)=P (T)(n\alpha w,z)= nP (T)(\alpha w,z).$$
Thus we have $P (T)(\alpha w,z)= \alpha P (T)(w,z)$ for any $\alpha \in \mathbb{Q}$. Again suppose $r \in \mathbb{R}$, then choose $(r_n) \subset \mathbb{Q}$ such that $r_n$ converges to $r$.

Next for any $(x,y) \in X \times Y$
\begin{eqnarray*}
    T_{rw,z}(x,y) &=& T(rw+x, y+z)\\
    &=& \lim_{n \to \infty} T(r_n w+x, y+z)   \ (\mbox{as} \ T \ \mbox{is Lipschitz in both co-ordinates})\\
    &=& \lim_{n \to \infty} T_{r_n w,z}(x,y).
\end{eqnarray*}
Further using continuity of $\mathcal{M}$ we get $$\mathcal{M}\left(\phi_T^{rw,z}\right)= \lim_{n \to \infty} \mathcal{M}\left(\phi_T^{r_n w,z}\right) = \lim_{n \to \infty} P(T)(r_n w,z)= \lim_{n \to \infty} r_n P(T)(w,z)= rP(T)(w,z).$$
Therefore, $P(T)$ is linear in the first variable. Likewise, it can also be shown to be linear in the second variable. Thus the map $$P: (BLip_0(X,Y; E),BLip(\cdot)) \xrightarrow{} (Blin(X,Y; E),\|\cdot\|)$$ given by $$P (T)(w,z) = \mathcal{M}\left(\phi_T^{w,z}\right)$$ for each $T \in BLip_0(X,Y; E)$ and  $(w,z) \in X \times Y$; is well defined with 
$$\|P (T)(w,z)\| = \|\mathcal{M}\left(\phi_T^{w,z}\right)\| \leq \|\phi_T^{w,z}\|_{\infty} \leq BLip(T)\|w\|\|z\|.$$
 This shows that $\|P(T)\| \leq BLip(T)$. Again, from the linearity of $\mathcal{M}$, it easily follows that $P$ is linear.
 We show that $P(T)=T$ for any $T \in Blin(X, Y; E)$. Suppose $T \in Blin(X, Y; E)$. Then for any $(x, y) \in X \times Y$
 $$\phi_T^{w,z}(x,y)= T(w+x,z+y) - T(w+x,y) -T(x,z+y) + T(x,y)=T(w,z).$$
 Moreover $\phi_T^{w,z}= \widehat{T(w,z)}$ and hence $P(T)(w,z)=\mathcal{M}(\phi_T^{w,z})=\mathcal{M}(\widehat{T(w,z)})= T(w,z)$. 

 Therefore, $P(T) = T$ for all $T \in Blin(X, Y; E)$ so that $P$ is a norm one linear projection from $BLip_0(X, Y; E)$ onto $Blin(X, Y; E)$. In other words, $Blin(X, Y; E)$ is complemented in $BLip_0(X, Y; E)$ whenever $E$ is a dual Banach space. 
\end{proof}
\begin{corollary}
	Let $X$ and $Y$ be Banach spaces with $E$ be a dual space. Then $ BLip_0(X, Y; E)$ is topologically isomorphic to $Blin(X, Y; E) \oplus_1 \ker(P)$.
\end{corollary}
\begin{proof}
By Theorem \ref{proj}, $P$ is a (norm one) projection with $P(BLip_0(X, Y; E)) = Blin(X, Y; E)$. Thus $BLip_{0}(X, Y; E) = Blin(X, Y; E) \oplus \ker(P)$ as vector spaces. Here the linear isomorphism is given by $$f \in BLip_0(X, Y: E) \mapsto  (P(T), (T - P(T))) \in Blin(X, Y: E) \oplus \ker(P).$$ 
Moreover, $\|\cdot\|_1$ is equivalent to $BLip(.)$ on $BLip_0(X, Y; E)$, where 
$$\|T\|_1 = \|P(T)\| + BLip(T - P(T))$$ 
for all $T \in BLip_0(X, Y;E)$. In fact, 
$$BLip(T) = \Vert P(T) \Vert + BLip(T - P(T)) \leq BLip(T) + 2 \Vert P(T) \Vert \leq 3 BLip(T),$$  
if $T \in BLip_{0}(X, Y; E)$.
\end{proof}
\begin{remark} \label{1st quo}
    We note that $\ker(P)$ is topologically isomorphic to $\bigslant{BLip_{0}(X, Y; E)}{Blin(X, Y; E) }$. In fact, $Id_{BLip_0(X,Y; E)} - P$ is a bounded linear surjective projection form $BLip_0(X,Y; E)$ onto  $\ker(P)$ with kernel $Blin(X, Y; E)$. 
\end{remark} 
\section{Vector valued dual action} \label{sec3}
Here, we define an $E$-valued dual action between 
$BLip_0(X, Y; E)$ and 
$F(X)\hat{\otimes}_{\pi}F(Y)$, which will serve as a foundational setup for the results presented in the following sections.
 Let us first recall the following two theorems:
\begin{theorem} \cite[Theorem~2.6]{TLOI} 
For every \emph{two-Lipschitz} operator $T \in BLip_0(X, Y; E)$ there exists a unique bilinear mapping $T_B : F(X) \times F(Y) \xrightarrow{} E$ satisfying $T_B(\delta_x, \delta_y) = T(x,y)$ and
$T= T_B \circ (\delta_X , \delta_Y ) : X \times Y \xrightarrow{(\delta_X , \delta_Y)} F(X) \times F(Y) \xrightarrow{T_B} E$.
Furthermore $BLip (T) = \|T_B\|$. The bilinear mapping $T_B$ is called bi-linearization of the \emph{two-Lipschitz}
operator $T$.

This result implies that the space $(BLip_0(X, Y; E), BLip(\cdot))$ is isometrically isomorphic to $(Blin(F(X), F(Y); E), \|\cdot\|)$
\end{theorem}
\begin{theorem} $($\cite[Theorem~2.9]{ITTPOBS}, \cite[Remark~2.7]{TLOI}$)$
    The bilinear operator $T_B$ admits a linearization
$(T_B)_L: F(X)\hat{\otimes}_{\pi}F(Y) \xrightarrow{} E$ 
satisfies
$T= T_B \circ (\delta_X , \delta_Y ) =  (T_B)_L \circ\sigma_2 \circ (\delta_X , \delta_Y )$,
where $\sigma_2: F(X) \times F(Y) \xrightarrow{}F(X)\hat{\otimes}_{\pi}F(Y)$ is the canonical bilinear operator defined by $\sigma_2(\mu,\nu)=\mu \otimes \nu$. In addition, we have
$BLip(T) = \|T_B\|= \|(T_B)_L\|$
The linear operator $(T_B)_L$ is referred to as the linearization of the \emph{two-Lipschitz} operator $T$. For the
simplification, we write $T_L$ instead of $(T_B)_L$.
\end{theorem}
By combining these two results, we conclude that the space $(BLip_0(X, Y; E), BLip(\cdot))$ is isometrically isomorphic to $L\left(F(X)\hat{\otimes}_{\pi}F(Y), E\right)$ via the correspondence $T \mapsto T_L$. This encourages us to consider the following vector-valued duality, which will play a crucial role in the results of this section.

 \begin{proposition}
    For any Banach spaces $X, Y$ and $E$, define
       $$\langle \cdot,\cdot \rangle_{E} : BLip_{0}(X, Y; E) \times F(X)\hat{\otimes}_{\pi}F(Y) \xrightarrow{} E$$ given by $$\langle T,\xi \rangle_{E}:= T_L(\xi)$$ for all $T \in BLip_{0}(X, Y; E)$ and $\xi \in F(X)\hat{\otimes}_{\pi}F(Y)$. Then $\langle \cdot,\cdot \rangle_{E}$ is bilinear and non-degenerate.
       \end{proposition}
       \begin{proof}
            The existence of a unique map $T_L$, corresponding to $T\in BLip_{0}(X, Y; E)$, provides that $\langle \cdot,\cdot \rangle_{E}$ is linear in the first coordinate, and the linearity of $T_L$ gives the linearity in the second coordinate of $\langle \cdot,\cdot \rangle_{E}$.
       Also the following norm conditions hold:$$\sup\limits_{\xi \in F(X)\hat{\otimes}_{\pi}F(Y), \|\xi\|\leq 1} \|T_L(\xi)\|= \|T_L\|= BLip(T)\ \mbox{and} \ \sup\limits_{BLip(T)\leq 1} \|T_L(\xi)\|=\sup\limits_{\|T_L\|\leq 1} \|T_L(\xi)\|=\|\xi\|. $$
Now we show that $\langle \cdot,\cdot \rangle_{E}$ is non-degenerate.
    Let $T \in BLip_{0}(X, Y; E)$ and $T_L(\xi)=0$ for all $\xi \in F(X)\hat{\otimes}_{\pi}F(Y)$. This gives $T_L(\delta_{x} \otimes \delta_y)=0$ for all $x \in X$ and $y \in Y$. That is, $T(x,y)=0$ for all $x \in X$ and $y \in Y$. Hence $T =0$.
    
     Again let $\xi \in F(X)\hat{\otimes}_{\pi}F(Y)$ be such that $T_L(\xi) =0$ for all $T \in BLip_{0}(X, Y; E)$. That is $T_L(\xi) =0$ for all $T_L \in L\left(F(X)\hat{\otimes}_{\pi}F(Y), E\right)$. Then by Hahn-Banach theorem there exists $S \in L\left(F(X)\hat{\otimes}_{\pi}F(Y), \mathbb{R}\right)$ such that $S(\xi) = \|\xi\|$. Now define $\tilde{S} (\gamma) = S(\gamma) e$ for all $\gamma \in F(X)\hat{\otimes}_{\pi}F(Y)$ and some unit vector $e \in E$. Then $\tilde{S} \in L\left(F(X)\hat{\otimes}_{\pi}F(Y), E\right)$ and hence it follows that $ \tilde{S}(\xi)=0$. That is $ \|\xi \|=0$. Thus, $\xi =0$. This completes the proof.
\end{proof}
For $\mathcal{A} \subset BLip_{0}(X, Y; E)$ we define $\mathcal{A}^{\diamondsuit}:=\{\xi \in F(X)\hat{\otimes}_{\pi}F(Y): T_L(\xi)=0 \ \ \forall ~T \in \mathcal{A}\}$. Then $\mathcal{A}^{\diamondsuit} = \bigcap\limits_{T \in \mathcal{A}} \ker(T_L)$ is a closed subspace of $F(X)\hat{\otimes}_{\pi}F(Y)$.

Further, for $\mathcal{D} \subset F(X)\hat{\otimes}_{\pi}F(Y)$ we consider the space $$ ^{\diamondsuit}\mathcal{D}:= \{ T \in BLip_{0}(X, Y; E) : T_L(\xi)=0 ~ \forall \xi \in \mathcal{D}\}.$$ Using the fact $BLip(T) = \|T_L\|$, we can easily verify that it is also a closed subspace of $BLip_{0}(X, Y; E)$.

\section{some quotient space}\label{sec4}
Building on the dual action established in the preceding section, we present some of the main results, including identifying the quotient space in this context. Additionally, we provide a necessary and sufficient condition for a \emph{two-Lipschitz} map to be bilinear.
\begin{proposition} \label{quotient of Blip}
    Let $\mathcal{D}$ be a subspace of $F(X)\hat{\otimes}_{\pi}F(Y)$ and $E$ be an injective Banach space. Then $\bigslant{BLip_{0}(X, Y;E)}{^{\diamondsuit}\mathcal{D} } \cong  L(\mathcal{D}, E)$.
\end{proposition}
\begin{proof}
For a fix $T \in BLip_0(X, Y; E)$,  we define $\theta_T : \mathcal{D} \to E$ given by $\theta_T(\xi) = T_L(\xi)$, for all $\xi \in \mathcal{D}$. Then, $\theta_T$ is linear. Also, for any $\xi \in \mathcal{D}$, we have 
		$$\|\theta_T(\xi)\|=\|T_L(\xi)\|\leq \|T_L\|\|\xi\|= BLip(T)\|\xi\|.$$ 
	Thus $ \theta_T \in L(\mathcal{D}, E)$  for any $T \in BLip_0(X, Y; E)$. 

We define $\Theta : BLip_0(X, Y; E) \to L(\mathcal{D}, E)$ given by $\Theta(T)= \theta_T$. It is routine to check that $\Theta$ is linear. We show that $\Theta$ is also surjective. Let $S \in L(\mathcal{D}, E)$.  Since $E$ is injective there exists $\widetilde{S} \in L(F(X)\hat{\otimes}_{\pi}F(Y), E)$ such that $\widetilde{T}\big|_{\mathcal{D}}=S$ with $\|\widetilde{S}\|=\|S\|$. Thus $\widetilde{S} \circ \sigma_2 \circ (\delta_X, \delta_Y) \in BLip_0(X, Y; E)$. Furthermore, the existence of a unique linearisation map corresponding to 
$T$ implies that $T_L= \widetilde{S}$. Thus, $\Theta$ is surjective.

	Also, 
	\begin{eqnarray*}
		\ker(\Theta)&=&\left\{T\in BLip_{0}(X, Y; E): \theta_T=0 \right\}\\
		&=&\left\{T\in BLip_{0}(X, Y; E): 0 = \theta_T (\xi)= T_L(\xi) ~ \mbox{for all}~ \xi \in \mathcal{D}\right\}\\
		&=& \ ^{\diamondsuit}\mathcal{D}.
	\end{eqnarray*} 
	So by fundamental theorem of linear algebra $\Theta$ induce a linear isomorphism (which we again denote by $\Theta$): 
	$$\Theta : \bigslant{BLip_{0}(X, Y; E)}{^{\diamondsuit}\mathcal{D} } \to L(\mathcal{D}, E)$$
	given by $\Theta(T+ \ ^{\diamondsuit}\mathcal{D})(\xi)=T_L(\xi)$ for all $\xi \in \mathcal{D}$.
	
	Now, for any $T \in BLip_0(X, Y; E)$ and $\xi \in \mathcal{D}$
	\begin{eqnarray*} \label{cont of theta1}
		\|\Theta(T+ \ ^{\diamondsuit}\mathcal{D})(\xi)\| &=& \|T_L(\xi)\| \\ 
		&=& \inf\limits_{U \in \ ^{\diamondsuit}\mathcal{D} }\|(T+U)_L(\xi)\| \\ 
		&\le& \|\xi\|\inf\limits_{U \in \ ^{\diamondsuit}\mathcal{D} }BLip(T+U) \\ 
		&=& \|\xi\| \|T+ \ ^{\diamondsuit}\mathcal{D}\|.
	\end{eqnarray*}
	Therefore, $\Theta$ is a surjective contractive linear isomorphism. So, by the open mapping theorem, $\Theta^{-1}$ exists and is a bounded linear surjective isomorphism. 

	Fix $S \in L(\mathcal{D}, E)$ and let $S_1$ and $S_2$ be any two extensions of $S$ in $L(F(X)\hat{\otimes}_{\pi}F(Y), E)$. If $\xi \in \mathcal{D}$, then 
	$$\left(\Theta(S_1 \circ \sigma_2 \circ (\delta_X, \delta_Y)) - \Theta(S_2 \circ \sigma_2 \circ (\delta_X, \delta_Y))\right)(\xi) = (S_1)_L(\xi)-(S_2)_L(\xi)=S(\xi)-S(\xi)=0.$$ Thus $S_1 \circ \sigma_2 \circ (\delta_X, \delta_Y) + \ ^{\diamondsuit}\mathcal{D} = S_2 \circ \sigma_2 \circ (\delta_X, \delta_Y) + \ ^{\diamondsuit}\mathcal{D}$ and hence $\widetilde{S} \circ \delta_{X}^{Y}+ \ ^{\diamondsuit}\mathcal{D}$ is uniquely determined by $S$. Now it follows that $\Theta^{-1}(T)=\widetilde{T} \circ \sigma_2 \circ (\delta_{X}, \delta_{Y})+ \ ^{\diamondsuit}\mathcal{D}$ for all $T \in L(\mathcal{D}, E)$, where $\widetilde{T} \in L(F(X)\hat{\otimes}_{\pi}F(Y), E)$ is a norm preserving extension of $T$. 	Since 
	$$\|\Theta^{-1}(T)\|=\|\widetilde{T} \circ \sigma_2 \circ (\delta_{X}, \delta_{Y})+ \ ^{\diamondsuit}\mathcal{D}\|\leq BLip(\widetilde{T} \circ \sigma_2 \circ (\delta_{X}, \delta_{Y})) \leq \|T\|,$$ 
	$\Theta^{-1}$ is also a contraction. Now, it is easy to conclude that $\Theta$ is surjective linear isometry so that $\bigslant{BLip_{0}(X, Y; E)}{^{\diamondsuit}\mathcal{D} }$ is isometrically isomorphic to $L(\mathcal{D}, E)$.
\end{proof}
\begin{remark}
    In case if we consider $\mathcal{D}= F(X)\hat{\otimes}_{\pi}F(Y)$, we get back the result  
    $$(BLip_0(X, Y; E), BLip(.))\cong L\left(F(X)\hat{\otimes}_{\pi}F(Y), E\right).$$
\end{remark}
The following result states that every element of the Lipschitz free space can be represented as an infinite series. It is a modified proof in our set up whose proof follows similarly as \cite[Lemma 2.1]{SAEPILFS}. This fact will be used throughout the article without further reference.
\begin{lemma}
    Let $\mu \in F(X)$. Then for every $\epsilon >0$, there exist sequences $(a_n)$ in $\mathbb{R}$ and $(x_n)$ in $X$ such that $\mu = \sum\limits_{n=1}^{\infty}a_n \delta_{x_n}$ and $\sum\limits_{n=1}^{\infty}|a_n| \|\delta_{x_n}\| = \sum\limits_{n=1}^{\infty}|a_n| \|x_n\| < \|\mu\|+\epsilon$.
\end{lemma}
\begin{proof}
    Let $\mu \in F(X)$ and $\epsilon > 0$. By Lemma \cite[3.100]{BST} there exists $(\mu_n) \subset span\left\{\delta_x: x \in X\right\}$ such that $\mu = \sum\limits_{n=1}^{\infty} \mu_n$ and $$\sum\limits_{n=1}^{\infty} \|\mu_n\| < \|\mu\| + \frac{\epsilon}{2}.$$
    Since $\mu_n \in F(X)=\overline{span\left\{\delta_x: x \in X\right\}}$, for each $n \in \mathbb{N}$ we can find a representation 
    $$\mu_n = \sum\limits_{i=1}^{I_n} a_i^n \delta_{x_i^n}$$ such that $\sum\limits_{i=1}^{I_n} |a_i^n| \|x_i^n\| < \|\mu_n\| + \frac{\epsilon}{2^{n+1}}$, where $I_n \in \mathbb{N}$.
    We re-index the sequences $(a_i^n)_{n,i}, (x_i^n)_{n,i}$ as $(a_j)_{j=1}^{\infty}, (x_j)_{j=1}^{\infty}$, respectively. Then
    $$\sum\limits_{j=1}^{\infty} |a_j| \|x_j\| = \sum\limits_{n,i}^{}|a_i^n| \|x_i^n\|= \sum\limits_{n=1}^{\infty}\sum\limits_{i=1}^{I_n}|a_i^n| \|x_i^n\| \leq \sum\limits_{n=1}^{\infty} \|\mu_n\| + \frac{\epsilon}{2^{n+1}} < \|\mu\| + \epsilon.$$
    Therefore, the series $\sum\limits_{j=1}^{\infty} a_j \delta_{x_j}$ is absolutely convergent in the Banach space $F(X)$ and hence $\sum\limits_{j=1}^{\infty} a_j \delta_{x_j} \in F(X)$. Now we show that $\mu = \sum\limits_{j=1}^{\infty} a_j \delta_{x_j}$.

    Let $t >0$. Since $\mu = \sum\limits_{n=1}^{\infty} \mu_n$ and $\sum\limits_{n=1}^{\infty} \|\mu_n\| < \|\mu\| + \frac{\epsilon}{2}$, there exist $N_1, N_2, N_3 \in \mathbb{N}$ such that $\|\mu-\sum\limits_{j=1}^{m} \mu_n\| <t$ for all $m \geq N_1$; $\|\mu_m\| <t$ for all $m \geq N_2$ and $\frac{\epsilon}{2^{m}} < t$ for all $m \geq N_3$. Put $N = max\{N_1, N_2, N_3\}$. Then, $max\left\{\|\mu-\sum\limits_{j=1}^{m} \mu_n\|, \|\mu_m\|, \frac{\epsilon}{2^{m}} \right\} < t$ for all $m \geq N$. Now for all $n > \sum\limits_{k=1}^{N} I_k$
\begin{eqnarray*}
    \left\|\mu-\sum\limits_{j=1}^{n} a_j \delta_{x_j}\right\| &\leq& \left\|\mu-\sum\limits_{j=1}^{m-1} \mu_j\right\| + \sum\limits_{i=1}^{I_m} |a_i^m| \|x_i^m\|  \\
    &\leq& \left\|\mu-\sum\limits_{j=1}^{m-1} \mu_j\right\| + \|\mu_m\|+ \frac{\epsilon}{2^m} < 3t, 
\end{eqnarray*}
where $m \in \mathbb{N}$ satisfies $\sum\limits_{k=1}^{m-1} I_k < n \leq \sum\limits_{k=1}^{m} I_k$.
    This completes the proof.
    \end{proof}
     The following result is key to  characterize the quotient $\bigslant{BLip_{0}(X, Y; E)}{Blin(X, Y; E) }$.
    \subsection{A necessary sufficient condition for a \emph{two-Lipschitz} map to be bilinear}
 \begin{proposition} \label{blip<-->blin}
    Let $T \in BLip_0(X, Y; E)$. Then $T$ is bilinear if and only if $T_L(\xi)=0$ for all $\xi \in \ker(\beta_X)\hat{\otimes}_{\pi}F(Y) + F(X)\hat{\otimes}_{\pi} \ker(\beta_Y)$.
\end{proposition}
\begin{proof}
    Put $\mathcal{R}= \ker(\beta_X)\hat{\otimes}_{\pi}F(Y) + F(X)\hat{\otimes}_{\pi} \ker(\beta_Y)$.
     Let $T \in BLip_{0}(X, Y; E)$ such that $T_L(\xi)=0$ for all $\xi \in \mathcal{R}$.
      We show that $T$ is bilinear.
     Suppose $r \in \mathbb{R}$, $x, x' \in X$ and $y \in Y$. Then $$(\delta_{rx+x'}- (r\delta_{x}+ \delta_{x'})) \otimes \delta_y, \delta_x \otimes (\delta_{ry+y'}- (r\delta_{y}+\delta_{y'}) \in \mathcal{R}.$$ Therefore $T_L\left((\delta_{rx+x'}- (r\delta_{x}+ \delta_{x'})) \otimes \delta_y \right)=0$ implies 
$T_L(\delta_{rx+x'} \otimes\delta_y)= rT_L(\delta_{x} \otimes\delta_y)+ T_L(\delta_{x'} \otimes\delta_y)$. Hence $T(rx+x',y)=r T(x,y)+ T(x', y)$.
Similarly we get $T(x,ry+y')=rT(x,y)+ T(x, y')$. Therefore, $T$ is bilinear.

Conversely, let $T \in Blin(X, Y; E)$ and $\xi \in \ker(\beta_{X})\otimes_{\pi}F(Y) + F(X)\hat{\otimes}_{\pi} \ker(\beta_Y)$. Then $\xi = \xi_1 + \xi_2$.   Consider $\xi_1= \left(\sum\limits_{n=1}^{\infty} a_{n} \delta_{x_{n}}\right) \otimes \left(\sum\limits_{m=1}^{\infty} b_{m} \delta_{y_{m}}\right)$  
 for $(x_{n}) \subset X$, $(y_m) \subset Y$ and $(a_{n}), (b_m) \subset \mathbb{R}$ with $\sum\limits_{n=1}^{\infty} a_{n} x_{n} =0$, and $\xi_2= \left(\sum\limits_{k=1}^{\infty} c_{k} \delta_{z_{k}}\right) \otimes \left(\sum\limits_{t=1}^{\infty} d_{t} \delta_{w_{t}}\right)$  
 for $(z_{k}) \subset X$, $(w_t) \subset Y$ and $(c_{k}), (d_t) \subset \mathbb{R}$ with $\sum\limits_{t=1}^{\infty} d_{t} w_{t} =0$.
 
 Therefore, using linearity and continuity of $T$ in both coordinates, we get 
 \begin{eqnarray*}
     T_L(\xi)&=& T_L(\xi_1)+ T_L(\xi_2)\\
     &=&\sum\limits_{n=1}^{\infty} a_{n} \sum\limits_{m=1}^{\infty} b_{m} T(x_n,y_m) + \sum\limits_{k=1}^{\infty} c_{k} \sum\limits_{t=1}^{\infty} d_{t} T(z_k,w_t)\\
     &=& \sum\limits_{m=1}^{\infty} b_{m} T\left(\sum\limits_{n=1}^{\infty} a_{n}x_n,y_m\right) + \sum\limits_{k=1}^{\infty} c_{k} T\left(z_k,\sum\limits_{t=1}^{\infty}d_{t}w_t\right)\\
     &=& \sum\limits_{m=1}^{\infty} b_{m} T\left(0,y_m\right) + \sum\limits_{k=1}^{\infty} c_{k} T\left(z_k,0\right)=0.
 \end{eqnarray*}
  Hence, the proof follows using the density argument.
\end{proof}
\begin{remark} \label{four rem }
\begin{enumerate}
    \item The above proposition concludes that $~^\diamondsuit \left(\ker(\beta_X)\hat{\otimes}_{\pi}F(Y) + F(X)\hat{\otimes}_{\pi} \ker(\beta_Y)\right)= Blin(X, Y; E)$. In fact it can be shown that $$~ ^{\diamondsuit}\left(\ker(\beta_X)\hat{\otimes}_{\pi}F(Y) \cup F(X)\hat{\otimes}_{\pi} \ker(\beta_Y)\right)= Blin(X, Y; E) = ~ ^{\diamondsuit}\left(\overline{\ker(\beta_X)\hat{\otimes}_{\pi}F(Y) + F(X)\hat{\otimes}_{\pi} \ker(\beta_Y)}^{\|\cdot\|_{\pi}}\right).$$ Since $\ker(\beta_X)\hat{\otimes}_{\pi}F(Y) \cup F(X)\hat{\otimes}_{\pi} \ker(\beta_Y)$ is not a subspace of $F(X)\hat{\otimes}_{\pi}F(Y)$, being a subspace $\ker(\beta_X)\hat{\otimes}_{\pi}F(Y) + F(X)\hat{\otimes}_{\pi} \ker(\beta_Y)$ is more useful.
    \item From the above Propositions \ref{quotient of Blip} and \ref{blip<-->blin} we get $$\bigslant{BLip_{0}(X, Y; E)}{Blin(X, Y; E) } \cong L\left(\ker(\beta_X)\hat{\otimes}_{\pi}F(Y) + F(X)\hat{\otimes}_{\pi} \ker(\beta_Y), E\right),$$ whenever $E$ is an injective Banach space.
\item Immediately we deduce that $$\left(\ ^{\diamondsuit}\left(\ker(\beta_X)\hat{\otimes}_{\pi}F(Y) \cup F(X)\hat{\otimes}_{\pi} \ker(\beta_Y)\right)\right)^\diamondsuit= Blin(X, Y; E)^\diamondsuit = \overline{\ker(\beta_X)\hat{\otimes}_{\pi}F(Y) + F(X)\hat{\otimes}_{\pi} \ker(\beta_Y)}.$$
\end{enumerate}
\end{remark}
Next, we see an auxiliary result related to the quotient using the duality notion.
\begin{proposition} \label{prop q.o free space}
     Let $\mathcal{D}$ be a closed subspace of $F(X)\hat{\otimes}_{\pi} F(Y)$, then $L\left(\bigslant{F(X)\hat{\otimes}_{\pi} F(Y)}{\mathcal{D} }, E\right)$ is isometrically isomorphic to $\ ^{\diamondsuit}\mathcal{D}$.
\end{proposition}
\begin{proof}
	For a fix $T \in~^{\diamondsuit}\mathcal{D}$, we define $\Lambda_T : F(X)\hat{\otimes}_{\pi} F(Y) \xrightarrow{} E$ given by $\Lambda_T(\xi)=T_L(\xi)$ for all $\xi \in F(X)\hat{\otimes}_{\pi} F(Y)$. Then, easily, we can verify that $\Lambda_T$ is a linear map with $\|\Lambda_T\| \leq BLip(T)$. Further
	$\ker(\Lambda_T) = \left\{ \xi \in F(X)\hat{\otimes}_{\pi} F(Y): T_L(\xi)=0\right\} \supset \mathcal{D}$. Therefore $\Lambda_T$ induces a bounded linear map (again denoted by $\Lambda_T$) from $\bigslant{F(X)\hat{\otimes}_{\pi} F(Y)}{\mathcal{D} }$ into $ E$. This induces a map 
	$$\Lambda : \ ^{\diamondsuit}\mathcal{D} \xrightarrow{} L\left(\bigslant{F(X)\hat{\otimes}_{\pi} F(Y)}{\mathcal{D} }, Y\right)$$ 
	given by $\Lambda(T)(\xi+\mathcal{D})=T_L(\xi)$ for all $T \in~^{\diamondsuit}\mathcal{D}$ and $\xi \in F(X)\hat{\otimes}_{\pi} F(Y)$. Then, $\Lambda$ is linear. Also 
	\begin{eqnarray*}
		\|\Lambda(T)(\xi+\mathcal{D})\| &=& \|T_L(\xi)\| \\ 
		&=& \inf\limits_{\mu \in \mathcal{D}}\|T_L(\xi+ \mu)\| \\ 
		&\le& \inf\limits_{\mu \in \mathcal{D}}\|(\xi+ \mu)\| BLip(T) \\ 
		&=& BLip(T)\|\xi + \mathcal{D}\|,
	\end{eqnarray*}
	for all $T \in~^{\diamondsuit}\mathcal{D}$ and $\xi \in F(X)\hat{\otimes}_{\pi} F(Y)$. Thus, $\Lambda$ is a contraction. Further, we show that $\Lambda$ is an isometry. 
	
Let $\epsilon >0$. Then there exist $x,x' \in X$ and $y,y' \in Y$ with $x\neq x', y \ne y'$ such that 
\begin{eqnarray*}
 \frac{\|T(x,y)-T(x',y)-T(x,y')+T(x',y')\|}{\|x-x'\|\|y-y'\|} &\geq& BLip(T)-\epsilon, \ \mbox{that is}   \\
 \left\|\frac{T_L(\delta_x\otimes\delta_y-\delta_x'\otimes\delta_y-\delta_x\otimes\delta_y'+\delta_x'\otimes\delta_y')}{\|x-x'\|\|y-y'\|}\right\|&\geq& BLip(T)-\epsilon. \  \mbox{Hence}\\
 \left\|T_L\left(\frac{\delta_x-\delta_{x'}}{\|x-x'\|} \otimes \frac{\delta_y-\delta_{y'}}{\|y-y'\|}\right)\right\| &\geq& BLip(T)-\epsilon
\end{eqnarray*}
Therefore, $$\left\|\Lambda(T)\left(\frac{\delta_x-\delta_{x'}}{\|x-x'\|} \otimes \frac{\delta_y-\delta_{y'}}{\|y-y'\|} + \mathcal{D}\right)\right\|=\frac{\|T(x,y)-T(x',y)-T(x,y')+T(x',y')\|}{\|x-x'\|\|y-y'\|} \geq BLip(T)-\epsilon.$$ 
	Since $1\geq\left\|\left(\frac{\delta_x-\delta_{x'}}{\|x-x'\|} \otimes \frac{\delta_y-\delta_{y'}}{\|y-y'\|}\right)\right\|\geq \left\|\frac{\delta_x-\delta_{x'}}{\|x-x'\|} \otimes \frac{\delta_y-\delta_{y'}}{\|y-y'\|} + \mathcal{D}\right\|$ and $\epsilon > 0$ is arbitrary, we conclude that $\Vert \Lambda(T) \Vert \ge BLip(T)$. Hence, $\Lambda$ is an isometry.
	
	To prove the surjectivity of $\Lambda$, let $S \in L\left(\bigslant{F(X)\hat{\otimes}_{\pi} F(Y)}{\mathcal{D} }, E\right)$. Set $T(x,y) = S(\delta_{x}\otimes \delta_y+\mathcal{D})$ for all $x \in X$ and $y \in Y$. Then it can be easily shown that $T \in BLip_0(X, Y; E)$ with $T_L= S$. Further for any $\xi \in ~^{\diamondsuit}\mathcal{D}$, $ T_L(\xi)=S(\xi +\mathcal{D})=0$. Hence $T \in ~^{\diamondsuit}\mathcal{D}$. This completes the proof. 
\end{proof}
\begin{remark}
\begin{enumerate}
\item Considering $\mathcal{D}=\overline{\ker(\beta_X)\hat{\otimes}_{\pi}F(Y) + F(X)\hat{\otimes}_{\pi} \ker(\beta_Y)}^{\|\cdot\|_{\pi}}$, using the Proposition \ref{prop q.o free space} and Remark \ref{four rem }$(1)$ we get $$Blin(X, Y; E) \cong L\left(\bigslant{F(X)\hat{\otimes}_{\pi} F(Y)}{\overline{\ker(\beta_X)\hat{\otimes}_{\pi}F(Y) + F(X)\hat{\otimes}_{\pi} \ker(\beta_Y)}^{\|\cdot\|_{\pi}} }, E\right),$$ whenever $E$ is an injective Banach space.
    \item In particular for $\mathcal{D}=\{0\}$ we recover $\left(BLip_0(X, Y; E), BLip(\cdot)\right)$ is isometrically isomorphic to $L\left(F(X)\hat{\otimes}_{\pi}F(Y), E\right)$.
\end{enumerate}
     \end{remark}
\section{Example}\label{sec5}
In this section, we aim to provide a nontrivial example of the quotient space, specifically
$$\bigslant{BLip_{0}(\mathbb{R}, \mathbb{R};\mathbb{R})}{Blin(\mathbb{R},\mathbb{R}; \mathbb{R}) }.$$ We determine it's predual explicitly using the already established results and the following Lemmas. 
\begin{lemma}
    $\ker(\beta_{\mathbb{R}})$ is isometrically isomorphic to $\{f \in L^{1}(\mathbb{R}) : \int_{\mathbb{R}}^{} f  dx=0\}$, where $dx$ is the Lebesgue measure on $\mathbb{R}$.
\end{lemma}
\begin{proof}
    Let us define $\Lambda :  L^1(\mathbb{R}) \xrightarrow{} \mathbb{R}$ given by 
    $$ \Lambda(f) = \int_{\mathbb{R}} f dx.$$
    Then, $\Lambda$ is a surjective linear contraction.

    Again recall (for details see \cite[p 542]{IROLFSOCDIDS}) the linear isometric isomorphism between $F(\mathbb{R})$ and $L^1(\mathbb{R})$ say
    $\phi:F(\mathbb{R}) \xrightarrow{} L^1(\mathbb{R})$ whose action on the spanning elements is as follows: 
    \[
		\phi(\delta_x) =
		\begin{cases}
        - \chi_{(x,0)} &\text{;if } x < 0\\
        0 &\text{;if } x = 0 \\
			\chi_{(0,x)} &\text{;if } x > 0. 
		\end{cases}
		\]
         In fact, from the proof, it follows that $\mathcal{R}= span\{\delta_x: x \in \mathbb{R}\}$ is linearly isometrically isomorphic to the space of all simple functions $\equiv \mathcal{S}$ $\left(\subset L^1(\mathbb{R})\right)$.
         
 Let $x \in \mathbb{R}$. Then 
 $$
		\Lambda \circ \phi (\delta_x) =
		\begin{cases}
        \Lambda(- \chi_{(x,0)})=x=\beta_{\mathbb{R}}(\delta_x) &\text{;if } x < 0\\
        0 &\text{;if } x = 0 \\
			\Lambda(\chi_{(0,x)})=x =\beta_{\mathbb{R}}(\delta_x)&\text{;if } x > 0. 
		\end{cases}
		$$
  Thus $\beta_{\mathbb{R}} = \Lambda \circ \phi$.
Further 
 \begin{eqnarray*}
     \ker(\beta_{\mathbb{R}}) &=& \left\{\gamma \in F(\mathbb{R}): \beta_{\mathbb{R}}(\gamma)=0\right\}\\
     &=& \left\{f \in L^1(\mathbb{R}): \beta_{\mathbb{R}}\left(\phi^{-1}(f)\right)=0\right\}\\
     &=& \left\{f \in L^1(\mathbb{R}): \Lambda(f)=0\right\}\\
     &=&\{f \in L^{1}(\mathbb{R}) : \int_{\mathbb{R}}^{} f  dx=0\}.
 \end{eqnarray*}
 This completes the proof.
\end{proof}
In case of $X=Y=\mathbb{R}$, $F(\mathbb{R})$ is isometrically isomorphic to $L^1(\mathbb{R})$(see \cite[p 128]{LFBS}), and $\ker(\beta_{\mathbb{R}}) \cong \left\{f \in L^1(\mathbb{R}): \int\limits_{\mathbb{R}} f(x)dx = 0\right\}$, is a closed subspace of $F(\mathbb{R})$. Also $\ker(\beta_\mathbb{R})\hat{\otimes}_{\pi}F(\mathbb{R}) + F(\mathbb{R})\hat{\otimes}_{\pi} \ker(\beta_\mathbb{R})$ is a subspace of $F(\mathbb{R})\hat{\otimes}_{\pi} F(\mathbb{R})$.

Put $\mathcal{D}=\ker(\beta_\mathbb{R})\hat{\otimes}_{\pi}F(\mathbb{R}) + F(\mathbb{R})\hat{\otimes}_{\pi} \ker(\beta_\mathbb{R})$. Then
 from the Proposition \ref{quotient of Blip} it follows 
    \begin{eqnarray*}
        \bigslant{BLip_{0}(\mathbb{R}, \mathbb{R};\mathbb{R})}{^{\diamondsuit}\mathcal{D} } &\cong&  L(\mathcal{D}, \mathbb{R}), \ \mbox{that is}\\
         \bigslant{BLip_{0}(\mathbb{R}, \mathbb{R};\mathbb{R})}{Blin(\mathbb{R},\mathbb{R}; \mathbb{R}) } &\cong&  L(\ker(\beta_\mathbb{R})\hat{\otimes}_{\pi}F(\mathbb{R}) + F(\mathbb{R})\hat{\otimes}_{\pi} \ker(\beta_\mathbb{R}), \mathbb{R}).
    \end{eqnarray*}
    In fact $BLip_{0}(\mathbb{R}, \mathbb{R};\mathbb{R}) \cong L(F(\mathbb{R})\hat{\otimes}_\pi F(\mathbb{R}), \mathbb{R}) \cong L^{\infty}(\mathbb{R} \times \mathbb{R}) \ (\mbox{by Exercise 2.8}$ \cite[p 43]{ITTPOBS}), where on $\mathbb{R} \times \mathbb{R}$ we consider the product measure. Hence we further deduce that $$\bigslant{L^{\infty}(\mathbb{R}\times \mathbb{R})}{span\{\textbf{1}\}} \cong  L(\ker(\beta_\mathbb{R})\hat{\otimes}_{\pi}F(\mathbb{R}) + F(\mathbb{R})\hat{\otimes}_{\pi} \ker(\beta_\mathbb{R}), \mathbb{R}).$$
    
By a standard measure-theoretic argument, it is known that $$\bigslant{L^{\infty}(\mathbb{R}\times \mathbb{R})}{span\{\textbf{1}\}} \cong \left\{f \in L^1(\mathbb{R}^2): \int\limits_{\mathbb{R}^2} f dxdy =0\right\}^{\ast},$$ where $dxdy$ is the product measure on $\mathbb{R}^2$. Moreover we showed that $\ker(\beta_\mathbb{R})\hat{\otimes}_{\pi}F(\mathbb{R}) + F(\mathbb{R})\hat{\otimes}_{\pi} \ker(\beta_\mathbb{R})$ is another predual, explicitly described in the following. 
\begin{lemma}
    $\ker(\beta_\mathbb{R})\hat{\otimes}_{\pi}F(\mathbb{R})$ is isometrically isomorphic to $\left\{f \in L^1(\mathbb{R}^2): \int\limits_{\mathbb{R}} f(x,y) dx=0,~ \mbox{a.e}~ y\right\}$.
\end{lemma}
\begin{proof}
Let us denote $L_0^1(\mathbb{R}):= \left\{f \in L^1(\mathbb{R}): \int\limits_{\mathbb{R}} f(x)dx = 0\right\} \cong \ker(\beta_{\mathbb{R}})$.
   Since $F(\mathbb{R})\hat{\otimes}_{\pi}F(\mathbb{R}) \cong L^1(\mathbb{R})\hat{\otimes}_{\pi} L^1(\mathbb{R})$, we consider the canonical isometric isomorphism (see $253F$, Theorem \cite[p 230]{MTV2}) $$\Phi: L^1(\mathbb{R})\hat{\otimes}_{\pi} L^1(\mathbb{R}) \to L^1(\mathbb{R}^2) $$
   given by $\Phi(f \otimes g)(x,y)=f(x)g(y)$ for all $f,g \in L^1(\mathbb{R})$ and $x,y \in \mathbb{R}$. 
   Therefore, $\ker(\beta_\mathbb{R})\hat{\otimes}_{\pi}F(\mathbb{R}) \cong L_0^1(\mathbb{R})\hat{\otimes}_{\pi}L^1(\mathbb{R})$. So it is enough to show that $L_0^1(\mathbb{R})\hat{\otimes}_{\pi}L^1(\mathbb{R})$ is isometrically isomorphic to $\left\{f \in L^1(\mathbb{R}^2): \int\limits_{\mathbb{R}} f(x,y) dx=0,~ \mbox{a.e}~ y\right\}$.
   
   Put $V= \Phi\left(L_0^1(\mathbb{R})\hat{\otimes}_{\pi}L^1(\mathbb{R})\right)= \overline{span\{f(x)g(y) : f \in L_0^1(\mathbb{R}), g \in L^1(\mathbb{R})\}}^{\|\cdot\|_1} \subset L^1(\mathbb{R}^2)$, and $W= \left\{f \in L^1(\mathbb{R}^2): \int\limits_{\mathbb{R}} f(x,y) dx=0,~ \mbox{a.e}~ y\right\}$. It is immediate that $V \subset W$. Suppose $F \in W$. Then there exists $(F_n) \subset L^1(\mathbb{R}^2)$, a sequence of finite sums of elementary $f \otimes g$; approximating $F$ in $L^1$ norm. Choose $\psi \in C_c(\mathbb{R})$ (space of all compactly supported continuous functions) with $\int\limits_{\mathbb{R}} \psi dx =1$ and consider
   $$\tilde{F_n}(x,y)= F_n(x,y) - \left(\int\limits_{\mathbb{R}}F_n(z,y) dz\right)\psi(x).$$ 
   Then $\int\limits_{\mathbb{R}}\tilde{F_n}(x,y) dx=0$. Now we show that $\tilde{F_n} \in span\{f(x)g(y) : f \in L_0^1(\mathbb{R}), g \in L^1(\mathbb{R})\}$. 
   
   Set $c_n(y)= \int\limits_{\mathbb{R}}F_n(z,y) dz = \sum\limits_{i=1}^{k} \left(\int\limits_{\mathbb{R}} f_i(z) dz\right) g_i(y) \in L^1(\mathbb{R})$. Therefore,
   \begin{eqnarray*}
       \tilde{F_n}(x,y) &=& \sum\limits_{i=1}^{k} f_i(x)g_i(y) - c_n(y) \psi(x) \\
       &=& \sum\limits_{i=1}^{k} f_i(x)g_i(y) - \psi(x) \sum\limits_{i=1}^{k} \left(\int\limits_{\mathbb{R}} f_i(z) dz\right) g_i(y)\\
       &=& \sum\limits_{i=1}^{k} \underbrace{\left(f_i(x) - \left(\int\limits_{\mathbb{R}} f_i(z) dz\right) \psi(x)\right)} g_i(y) \in span\left\{f(x)g(y) : f \in L_0^1(\mathbb{R}), g \in L^1(\mathbb{R})\right\}.
       \end{eqnarray*}
So it remains to prove that $\| \tilde{F}_n - F \|_{L^1(\mathbb{R}^2)} \to 0 $.

We estimate
$$
\| \tilde{F}_n - F \|_{L^1(\mathbb{R}^2)}
= \| F_n - c_n(y)\psi(x) - F \|_{L^1(\mathbb{R}^2)}
\leq \| F_n - F \|_{L^1(\mathbb{R}^2)} + \| c_n(y)\psi(x) \|_{L^1(\mathbb{R}^2)}.$$

Since $ F_n \to F $ in $L^1$ norm, we have
$
\| F_n - F \|_{L^1(\mathbb{R}^2)}$ converges to $0$ as $n \to \infty$.
Now  
$$
\| c_n(y)\psi(x) \|_{L^1(\mathbb{R}^2)}
= \int_{\mathbb{R}} \int_{\mathbb{R}} |c_n(y)\psi(x)|\,dx\,dy
= \|\psi\|_{L^1(\mathbb{R})} \cdot \|c_n\|_{L^1(\mathbb{R})}.
$$
Further we estimate $ \|c_n\|_{L^1(\mathbb{R})}$. Using the assumption that $F \in W$ that is $ \int F(x, y)\,dx = 0$, we obtain
$$
|c_n(y)| = \left| \int_{\mathbb{R}} F_n(x, y)\,dx \right|
= \left| \int_{\mathbb{R}} \left(F_n(x, y) - F(x, y)\right)\,dx \right|
\leq \int_{\mathbb{R}} |F_n(x, y) - F(x, y)|\,dx.
$$
Now, integrating over $y$, we get
$$
\| c_n \|_{L^1(\mathbb{R})}
= \int_{\mathbb{R}} |c_n(y)|\,dy
\leq \int_{\mathbb{R}} \int_{\mathbb{R}} |F_n(x, y) - F(x, y)|\,dx\,dy
= \| F_n - F \|_{L^1(\mathbb{R}^2)}.
$$
Thus
$
\| c_n(y)\psi(x) \|_{L^1} = \|\psi\|_{L^1(\mathbb{R})} \|c_n\|_{L^1(\mathbb{R})} $ converges to $0$. Hence, the proof follows.
\end{proof}
\begin{remark}
    Similarly, we can prove that $F(\mathbb{R})\hat{\otimes}_{\pi}\ker(\beta_\mathbb{R})$ is isometrically isomorphic to $$\left\{f \in L^1(\mathbb{R}^2): \int\limits_{\mathbb{R}} f(x,y) dy=0,~ \mbox{a.e}~ x\right\}.$$
    
    Hence we get 
    \begin{eqnarray*}
        \ker(\beta_\mathbb{R})\hat{\otimes}_{\pi}F(\mathbb{R}) + F(\mathbb{R})\hat{\otimes}_{\pi} \ker(\beta_\mathbb{R}) &\cong& \left\{f \in L^1(\mathbb{R}^2): \int\limits_{\mathbb{R}} f(x,y) dx=0,~ \mbox{a.e}~ y\right\}\\
        &+&\left\{f \in L^1(\mathbb{R}^2): \int\limits_{\mathbb{R}} f(x,y) dy=0,~ \mbox{a.e}~ x\right\}
    \end{eqnarray*}
    
\end{remark}
{\bf Acknowledgment.} The author gratefully acknowledges Dr. Debkumar Giri for his valuable suggestions that helped to shape the article in its present form.

{\bf Declaration.} 
The author was financially supported by the Senior Research Fellowship from the National Institute of Science Education and Research, Bhubaneswar, funded by the Department of Atomic Energy, Government of India. The author has no competing interests or conflicts of interest to declare relevant to this article's content.
%\bibliographystyle{plain}
%  This inserts the bib file, biblio.bib
%\bibliography{reference}

%\bibliography{reference}

\end{document}